\documentclass{amsart}

\usepackage{color,graphicx,amssymb,latexsym,amsfonts,txfonts,amsmath,amsthm}
\usepackage{pdfsync,MnSymbol,yfonts}

\newcommand{\s}{\vspace{0.3cm}}

\newtheorem{theo}{Theorem}[section]
\newtheorem{coro}[theo]{Corollary}
\newtheorem{lemm}[theo]{Lemma}
\newtheorem{prop}[theo]{Proposition}
\theoremstyle{remark}
\newtheorem{rema}[theo]{\bf Remark}

\numberwithin{equation}{section}
\newcommand{\blankbox}[2]{
\parbox{\columnwidth}{\centering}}

\begin{document}

\title{About the Fricke-Macbeath curve}
\author{Ruben A. Hidalgo}
\address{Departamento de Matem\'atica y Estad\'{\i}stica, Universidad de La Frontera, Temuco, Chile}
\email{ruben.hidalgo@ufrontera.cl}

\thanks{Partially supported by Project Fondecyt 1150003 and Project Anillo ACT 1415 PIA-CONICYT}

\subjclass[2010]{14H55, 30F10, 14H37, 14H40}
\keywords{Riemann surface, Algebraic curve, Hurwitz curve, Automorphism, Jacobians}

\begin{abstract}
A closed Riemann surface of genus $g \geq 2$ is called a Hurwitz curve if its group of conformal automorphisms has order $84(g-1)$. In 1895, A. Wiman noticed that  there is no Hurwitz curve of genus $g=2,4,5,6$ and, up to isomorphisms, there is a unique Hurwitz curve of genus $g=3$; this being Klein's plane quartic curve. Later, in 1965, A. Macbeath proved the existence, up to isomorphisms, of a unique Hurwitz curve of genus $g=7$; this known as the Fricke-Macbeath curve. Equations were also provided; that being the fiber product of suitable three elliptic curves. In the same year, W. Edge constructed such a genus seven Hurwitz curve by elementary projective geometry. Such a construction was provided by first constructing a $4$-dimensional family of genus seven closed Riemann surfaces $S_{\mu}$ admitting a group $G_{\mu} \cong {\mathbb Z}_{2}^{3}$ of conformal automorphisms so that $S_{\mu}/G_{\mu}$ has genus zero. In this paper we discuss the above curves in terms of fiber products of classical Fermat curves and we provide a geometrical explanation of the three elliptic curves in Macbeath's description. We also observe that the jacobian variety of each $S_{\mu}$ is isogenous to the product of seven elliptic curves (explicitly given) and, for the particular Fricke-Macbeath curve, we obtain the well known fact that its jacobian variety is isogenous to $E^{7}$ for a suitable elliptic curve $E$.

\end{abstract}

\maketitle

\section{Introduction}
 If $S$ is a closed Riemann surface of genus $g \geq 2$, then its group ${\rm Aut}(S)$, of conformal automorphisms, has order at most $84(g-1)$ (Hurwitz upper bound). One says that $S$ is a Hurwitz curve if $|{\rm Aut}(S)|=84(g-1)$; in this situation the quotient orbifold $S/{\rm Aut}(S)$ is the Riemann sphere $\widehat{\mathbb C}$ with exactly three cone points of respective orders $2,3$ and $7$, and 
$S={\mathbb H}^{2}/\Gamma$, where $\Gamma$ is a torsion free normal subgroup of finite index in the triangular Fuchsian group $\Delta=\langle x,y: x^2=y^3=(xy)^7=1\rangle$.
 
In 1895, A. Wiman \cite{Wiman} noticed that there is no Hurwitz curve in each genera $g \in\{2,4,5,6\}$ and that there is exactly one Hurwitz curve (up to isomorphisms) of genus three, this being given by Klein's quartic $x^3y+y^3z+z^3x=0$; whose group of conformal automorphisms  is the  simple group ${\rm PSL}_{2}(7)$ of order $168$. Later, in 1965,  A. Macbeath \cite{Macbeath} observed that in genus seven there is exactly one (up to isomorphisms) Hurwitz curve, called  the Fricke-Macbeath curve; its automorphisms group is the simple group ${\rm PSL}_{2}(8)$ of order $504$. In the same paper, an explicit equation for the Fricke-Macbeath curve over ${\mathbb Q}(\rho)$, where $\rho=e^{2 \pi i/7}$, is given as the fiber product of three (isomorphic) elliptic curves as follows

\begin{equation}\label{eq0}
\left\{\begin{array}{lcl}
y_{1}^{2}=(x-1)(x-\rho^3)(x-\rho^5)(x-\rho^6)\\
\\
y_{2}^{2}=(x-\rho^{2})(x-\rho^4)(x-\rho^5)(x-\rho^6)\\
\\
y_{3}^{2}=(x-\rho)(x-\rho^3)(x-\rho^4)(x-\rho^5)
\end{array}
\right\} \subset {\mathbb C}^{4}.
\end{equation}

In \cite{Prokhorov} Y. Prokhorov classified the finite simple non-abelian subgroups of the Cremona group of rank $3$, that is, the group of birational automorphisms of the $3$-dimensional complex projective space; these groups being isomorphic to ${\textswab A}_{5}$, ${\textswab A}_{6}$, ${\textswab A}_{7}$, ${\rm PSL}_{2}(7)$, ${\rm PSL}_{2}(8)$ and ${\rm PSp}_{4}(3)$ (where ${\textswab A}_{n}$ denotes the alternating group in $n$ letters). In his classification, the group ${\rm PSL}_{2}(8)$ is seen to act on some smooth Fano threefold in ${\mathbb P}_{\mathbb C}^{8}$ of genus $7$, this being the dual Fano threefold of the Fricke-Macbeath curve.

The uniqueness of the Fricke-Macbeath curve asserts, by results due to J. Wolfart \cite{Wolfart}, that it can be also represented by a curve over ${\mathbb Q}$. Recently, the following affine plane algebraic model  of the Fricke-Macbeath curve was obtained by Bradley Brock (personal communication) 
$$1 + 7xy + 21x^2 y^2 + 35x^3 y^3 + 28x^4 y^4 + 2x^7 + 2y^7 = 0,$$ 
and the following projective algebraic curve model in ${\mathbb P}^{6}$  by Maxim Hendriks in his Ph.D. Thesis \cite{Hendriks}  
$$\left\{ \begin{array}{rll}
-x_1 x_2 + x_1 x_0 + x_2 x_6 + x_3 x_4 - x_3 x_5 - x_3 x_0 - x_4 x_6 - x_5 x_6 &=&0\\
x_1 x_3 + x_1 x_6 - x_2^2 + 2 x_2 x_5 + x_2 x_0 - x_3^2 + x_4 x_5 - x_4 x_0 - x_5^2 &=&0\\
x_1^2 - x_1 x_3 + x_2^2 - x_2 x_4 - x_2 x_5 - x_2 x_0 - x_3^2 + x_3 x_6 + 2 x_5 x_0 - x_0^2 &=&0\\
x_1 x_4 - 2 x_1 x_5 + 2 x_1 x_0 - x_2 x_6 - x_3 x_4 - x_3 x_5 + x_5 x_6 + x_6 x_0 &=&0\\
x_1^2 - 2 x_1 x_3 - x_2^2 - x_2 x_4 - x_2 x_5 + 2 x_2 x_0 + x_3^2 + x_3 x_6 + x_4 x_5 + x_5^2 - x_5 x_0 - x_6^2 &=&0\\
x_1 x_2 - x_1 x_5 - 2 x_1 x_0 + 2 x_2 x_3 - x_3 x_0 - x_5 x_6 + 2 x_6 x_0 &=&0\\
-2 x_1 x_2 - x_1 x_4 - x_1 x_5 + 2 x_1 x_0 + 2 x_2 x_3 - 2 x_3 x_0 + 2 x_5 x_6 - x_6 x_0 &=&0\\
2 x_1^2 + x_1 x_3 - x_1 x_6 + 3 x_2 x_0 + x_4 x_5 - x_4 x_0 - x_5^2 + x_6^2 - x_0^2 &=&0\\
2 x_1^2 - x_1 x_3 + x_1 x_6 + x_2^2 + x_2 x_0 + x_3^2 - 2 x_3 x_6 + x_4 x_5 - x_4 x_0 + x_5^2 - 2 x_5 x_0 + x_6^2 + x_0^2 &=&0\\
x_1^2 + x_1 x_3 - x_1 x_6 + 2 x_2 x_5 - 3 x_2 x_0 + 2 x_3 x_6 + x_4^2 + x_4 x_5 - x_4 x_0 + x_6^2 + 3 x_0^2 &=&0
\end{array}
\right\}.
$$

\s

In 1965, W. Edge \cite{Edge} obtained the Fricke-Macbeath curve, via classical projective geometry, by using the fact that it admits a group $K = G \rtimes \langle T \rangle \cong {\mathbb Z}_{2}^{3} \rtimes {\mathbb Z}_{7}$ of conformal automorphisms. Edge first constructed a family of genus seven Riemann surfaces $S_{\mu}$, each one admitting the group $G_{\mu} \cong {\mathbb Z}_{2}^{3}$ as a group of conformal automorphisms so that $S_{\mu}/G_{\mu}$ has genus zero, and then he was able to isolate one of them in order to admit and extra order seven automorphism; this being the Fricke-Macbeath curve.

In this paper, we provide a different approach (but somehow related to Edge's) to obtain another set of equations and to explain, in a geometric manner,  the three elliptic curves appearing in equation \eqref{eq0}. Part of the results presented in here were already obtained in \cite{Hidalgo:FM}, where certain regular dessins d'enfants (Edmonds maps) were explicitly described over their minimal field of definition \cite{JSW}. We also describe explicitly an isogenous decomposition of the jacobian variety of $S_{\mu}$ as a product of seven (explicit) elliptic curves. Then we explicitly find the parameter $\mu=\mu^{0}$ so that $S_{\mu^{0}}$ is the Fricke-Macbeath curve and observe that these elliptic curves are isomorphic, obtaining the well known fact that its jacobian is isogenous to $E^{7}$, where $E$ is an elliptic curve.

\section{A $4$-dimensional family of genus seven Riemann surfaces}
As already noticed, the Fricke-Macbeath curve $S$ admits an abelian group $G \cong {\mathbb Z}_{2}^{3}$ as a group of conformal automorphism with quotient orbifold $S/G$ being the Riemann sphere (with exactly seven cone points of order two). We proceed to construct a $4$-dimensional of genus seven Riemann surfaces $S_{\mu}$, each one admitting a group $G_{\mu} \cong {\mathbb Z}_{2}^{3}$ as a group of conformal automorphisms so that $S_{\mu}/G_{\mu}$ has genus zero. Moreover, we will observe that each of these genus seven Riemann surfaces have a completely decomposable jacobian variety.

\subsection{A $4$-dimensional family of Riemann surfaces of genus $49$}
Let us consider the $4$-dimensional domain
$$\Omega:=\left\{\mu=(\mu_{4},\mu_{5},\mu_{6},\mu_{7}) \in {\mathbb C}^{4}: \mu_{j} \notin \{0,1\}, \; \mu_{i} \neq \mu_{j}, \; \mbox{for } i \neq j\right\} \subset {\mathbb C}^{4}.$$

As introduced in \cite{CGHR,GHL}, a closed Riemann surface $R$ is called a generalized Fermat curves of type $(2,6)$ if it admits a group $H \cong {\mathbb Z}_{2}^{6}$ of conformal automorphisms so that $R/H$ has genus zero. By the Riemann-Hurwitz formula, $R$ has genus $g=49$ and $R/H$ has exactly seven cone points, each one of order two. As we may identify $S/H$ with the Riemann sphere $\widehat{\mathbb C}$ (by the uniformization theorem) and any three different points can be sent to $\infty, 0, 1$ by a (unique) M\"obius transformation, we may assume these seven cone points to be $\infty$, $0$, $1$, $\mu_{4}^{R}$, $\mu_{5}^{R}$, $\mu_{6}^{R}$ and $\mu_{7}^{R}$, where $\mu^{R}:=(\mu_{4}^{R},\mu_{5}^{R},\mu_{6}^{R},\mu_{7}^{R}) \in \Omega$.

Now, for each $\mu:=(\mu_{4},\mu_{5},\mu_{6},\mu_{7}) \in \Omega$, we may consider the projective algebraic set 
\begin{equation}\label{eqgfc}
\widehat{S}_{\mu}=\left\{ \begin{array}{rclrcl}
x_{1}^{2}+x_{2}^{2}+x_{3}^{2}&=&0,&
\mu_{4} \; x_{1}^{2}+x_{2}^{2}+x_{4}^{2}&=&0,\\
\mu_{5} \; x_{1}^{2}+x_{2}^{2}+x_{5}^{2}&=&0,&
\mu_{6} \; x_{1}^{2}+x_{2}^{2}+x_{6}^{2}&=&0,\\
\mu_{7} \; x_{1}^{2}+x_{2}^{2}+x_{7}^{2}&=&0
\end{array}
\right\} \subset {\mathbb P}_{\mathbb C}^{6}.
\end{equation}

As the matrix
$$\left[ 
\begin{array}{rllllll}
2x_{1} & 2x_{2} & 2x_{3} & 0 & 0 & 0 & 0\\
2\mu_{4}x_{1} & 2x_{2} & 0 & 2x_{4} & 0 & 0 & 0\\
2\mu_{5}x_{1} & 2x_{2} & 0 &  0 & 2x_{5} & 0 & 0\\
2\mu_{6}x_{1} & 2x_{2} & 0 &  0 & 0 & 2x_{6} & 0 \\
2\mu_{7}x_{1} & 2x_{2} & 0 &  0 & 0 & 0 & 2x_{7}\\
\end{array}
\right]
$$
has maximal rank for $(x_{1},\ldots,x_{7}) \in {\mathbb C}^{7}-\{(0,\ldots,0)\}$ satisfying the above system of equations \eqref{eqgfc}, 
it turns out that $\widehat{S}_{\mu}$ is a smooth projective algebraic curve, so (by the implicit function theorem)  a closed Riemann surface.

If $j \in \{1,\ldots,7\}$, then denote by $a_{j} \in {\rm GL}_{7}({\mathbb C})$ the order two linear transformation defined 
by multiplication of the coordinate $x_{j}$ by $-1$.  It can be seen that each $a_{j}$ keeps $\widehat{S}_{\mu}$ invariant; so it induces a conformal automorphism of order two on it (which we will still be denoting as $a_{j}$). In this way, the group
$$H_{\mu}=\langle a_{1},\ldots,a_{6} \rangle \cong {\mathbb Z}_{2}^{6}$$
is a group of conformal automorphisms of $\widehat{S}_{\mu}$ and $a_{1}a_{2}a_{3}a_{4}a_{5}a_{6}a_{7}=1.$

Let $A$ be the M\"obius transformation so that $A(1)=\infty$, $A(\rho)=0$ and $A(\rho^{2})=1$, that is, 
$$A(x)=\frac{(1+\rho)(x-\rho)}{\rho(x-1)}.$$

It is not difficult to check that the map
$$P_{\mu}:\widehat{S}_{\mu} \to \widehat{\mathbb C}: 
[x_{1}: x_{2}: x_{3}: x_{4}: x_{5}: x_{6}:  x_{7}] \mapsto A^{-1}\left(-\left(\frac{x_{2}}{x_{1}}\right)^{2}\right)$$
is a regular branched cover, whose branch locus is the set 
$$B_{\mu}:=\{1, \rho, \rho^{2}, A^{-1}(\mu_{4}), A^{-1}(\mu_{5}),A^{-1}(\mu_{6}),A^{-1}(\mu_{7})\},$$
and whose deck group is $H_{\mu}$. In particular, $\widehat{S}_{\mu}$ is an example of a generalized Fermat curve of type $(2,6)$. 
In \cite{CGHR,GHL} it was observed that there is an isomorphism $\phi:R \to \widehat{S}_{\mu^{R}}$ conjugating $H$ to $H_{\mu^{R}}$.

As a consequence of the results in \cite{HKLP}, the group $H_{\mu}$ is a normal subgroup of ${\rm Aut}(\widehat{S}_{\mu})$ (in fact, in \cite{HKLP} it was proved that $H_{\mu}$ is the unique subgroup $L \cong {\mathbb Z}_{2}^{6}$ of ${\rm Aut}(\widehat{S}_{\mu})$ with quotient orbifold $\widehat{S}_{\mu}/L$ being of genus zero). The normality property of $H_{\mu}$ asserts the existence of a natural homomorphism
$$\widehat{\theta}_{\mu}:{\rm Aut}(\widehat{S}_{\mu}) \to {\rm PSL}_{2}({\mathbb C}): T \mapsto A_{T}$$
whose kernel is $H_{\mu}$ and satisfies
$$P_{\mu} \circ T= A_{T} \circ P_{\mu}, \; T \in {\rm Aut}(\widehat{S}_{\mu}).$$

\begin{rema}\label{observa1}
As a consequence of the Klein-Koebe-Poincar\'e uniformization theorem, there is a Fuchsian group 
$$\Gamma_{\mu}=\langle \alpha_{1},...,\alpha_{7}: \alpha_{1}^{2}=\cdots=\alpha_{7}^{2}=\alpha_{1} \alpha_{2} \cdots \alpha_{7}=1\rangle$$
acting on the hyperbolic plane ${\mathbb H}^{2}$ uniformizing the orbifold $\widehat{S}_{\mu}/H_{\mu}$. The derived subgroup $\Gamma_{\mu}'$ of $\Gamma_{\mu}$ is torsion free and $\widehat{S}_{\mu}={\mathbb H}^{2}/\Gamma_{\mu}'$; $H_{\mu}=\Gamma_{\mu}/\Gamma_{\mu}'$ (see, for instance, \cite{CGHR}). The generator $\alpha_{j}$ induces the involution $a_{j}$, for $j=1,\ldots,7$. This in particular asserts that the image of $\widehat{\theta}_{\mu}$ is exactly the stabilizer of the set $B_{\mu}$ in the group of M\"obius transformations.
\end{rema}


In \cite{CGHR} it was observed that the only non-trivial elements of $H_{\mu}$ acting with fixed points on $\widehat{S}$ are $a_{1},\ldots, a_{6}$ and $a_{7}$. This property, together the normality of $H_{\mu}$, asserts the existence of another natural homomorphism
$$\theta_{\mu}: {\rm Aut}(\widehat{S}_{\mu}) \to {\rm Aut}(H_{\mu})={\rm Aut}({\mathbb Z}_{2}^{6}): T \mapsto \Lambda_{T},$$
where 
$$\Lambda_{T}(a_{j})=T a_{j} T^{-1} \in \{a_{1},\ldots,a_{7}\}.$$

It can be seen that the kernel of $\theta_{\mu}$ is again $H_{\mu}$ (just note that if $T$ belongs to the kernel of $\theta_{\mu}$, then $A_{T}=I$ as it will necessarily fix each of the points in $B_{\mu}$).

\begin{rema}
Clearly, every automorphism of $H_{\mu}$ in the image of $\theta_{\mu}$ must keep invariant the set $\{a_{1},\ldots,a_{7}\}$. If 
$\eta \in {\rm Aut}(H_{\mu})$ is so that $\eta(a_{j}) \in \{a_{1},\ldots,a_{7}\}$, for every $j=1,\ldots,7$, then there is some orientation-preserving homeomorphism $F:\widehat{S}_{\mu} \to \widehat{S}_{\mu}$ normalizing $H_{\mu}$ and inducing $\eta$. Sometimes, one may chose $\mu$ in order to assume $F$ to be a conformal automorphism of $\widehat{S}_{\mu}$, but this is not in general true.
For instance, if $\eta(a_{1})=a_{2}$, $\eta(a_{2})=a_{3}$, $\eta(a_{3})=a_{1}$, $\eta(a_{4})=a_{5}$, $\eta(a_{5})=a_{6}$, $\eta(a_{6})=a_{7}$ and $\eta(a_{7})=a_{4}$, then the existence of 
$T \in {\rm Aut}(\widehat{S}_{\mu})$ so that $\eta=\Lambda_{T}$ will assert the existence of a M\"obius transformation $A_{T} \in {\rm PSL}_{2}({\mathbb C})$ such that 
 $A_{T}(1)=\rho$, $A_{T}(\rho)=\rho^{2}$, $A_{T}(\rho^{2})=1$, $A_{T}(A^{-1}(\mu_{4}))=A^{-1}(\mu_{5})$, $A_{T}(A^{-1}(\mu_{5}))=A^{-1}(\mu_{6})$, $A_{T}(A^{-1}(\mu_{6}))=A^{-1}(\mu_{7})$ and $A_{T}(A^{-1}(\mu_{7}))=A^{-1}(\mu_{4})$, which is clearly impossible. 
\end{rema}

If we set
$$A(\mu_{4},\mu_{5},\mu_{6},\mu_{7})=\left( \frac{1}{\mu_{4}},\frac{1}{\mu_{5}},\frac{1}{\mu_{6}},\frac{1}{\mu_{7}} \right),
B(\mu_{4},\mu_{5},\mu_{6},\mu_{7})=\left( \frac{\mu_{7}}{\mu_{7}-1},\frac{\mu_{7}}{\mu_{7}-\mu_{4}},\frac{\mu_{7}}{\mu_{7}-\mu_{5}},\frac{\mu_{7}}{\mu_{7}-\mu_{6}} \right),
$$
then ${\mathbb G}=\langle A, B\rangle \cong {\mathfrak S}_{7}$ is a group of holomorphic automorphisms of $\Omega$ and, 
in \cite{GHL} it was observed that, for $\mu, \widetilde{\mu} \in \Omega$, the Riemann surfaces $\widehat{S}_{\mu}$ and $\widehat{S}_{\widetilde{\mu}}$ are isomorphic if and only if they are ${\mathbb G}$-equivalent.
Let us also observe that, for each $\sigma \in {\rm Perm}\{4,5,6,7\} \cong {\mathfrak S}_{4}$,  $T_{\sigma} \in {\mathbb G}$, where
$$T_{\sigma}(\mu_{4},\mu_{5},\mu_{6},\mu_{7})=(\mu_{\sigma(4)},\mu_{\sigma(5)},\mu_{\sigma(6)},\mu_{\sigma(7)}).$$

\subsection{A $4$-dimensional family of genus seven Riemann surfaces}
For each $\mu \in \Omega$, let $\lambda, \zeta \in {\rm Aut}(H_{\mu})$ be defined as:
$$\lambda(a_{1})=a_{2},\; \lambda(a_{2})=a_{3},\; \lambda(a_{3})=a_{4},\; \lambda(a_{4})=a_{5},\; \lambda(a_{5})=a_{6},\; \lambda(a_{6})=a_{7},\; \lambda(a_{7})=a_{1},$$
$$\zeta(a_{1})=a_{2},\; \zeta(a_{2})=a_{1},\; \zeta(a_{3})=a_{7},\; \zeta(a_{7})=a_{3},\; \zeta(a_{4})=a_{6},\; \zeta(a_{6})=a_{4},\; \zeta(a_{5})=a_{5}.$$

The subgroup $\langle \lambda, \zeta\rangle$ of ${\rm Aut}(H_{\mu})$ is isomorphic to the dihedral group of order $14$. 

It can be checked (see also Section \ref{Sec:FM}) that $\lambda$ belongs the image of $\theta_{\mu^{0}}$, where $$\mu^{0}=(A(\rho^{3}),A(\rho^{4}),A(\rho^{5}),A(\rho^{6})) \in \Omega,$$
and $\zeta$ belongs to the image of $\theta_{(\mu_{4},\mu_{5},\mu_{6},\mu_{7})}$, where
$$\mu_{6}=\mu_{5}^{2}/\mu_{4}, \; \mu_{7}=\mu_{5}^{2}, \; \mu_{5}^{6}-5\mu_{5}^{4}-6\mu_{5}^{2}-1=0.$$

Note that $\mu^{0}$ satisfies these last conditions; so $\zeta$ also belongs to the image of $\theta_{\mu^{0}}$.

\begin{lemm}\label{rhoinv}
The only subgroups of $H_{\mu}$ which are $\lambda$-invariant are 
$$H_{\mu},\;
K_{\mu}=\langle a_{1}a_{3}a_{7}, a_{2}a_{3}a_{5},a_{1}a_{2}a_{4}\rangle\; and \;
K_{\mu}^{*}=\langle a_{1}a_{2}a_{6}, a_{2}a_{3}a_{7},a_{1}a_{3}a_{4}\rangle.$$

Moreover, the automorphism $\zeta$ permutes $K_{\mu}$ with $K_{\mu}^{*}$.
\end{lemm}
\begin{proof}
It is clear that the subgroups $H_{\mu}$, $K_{\mu}$ and $K_{\mu}^{*}$ are $\lambda$-invariants and that $\zeta$ permutes $K_{\mu}$ with $K_{\mu}^{*}$. We need to check that these are the only $\lambda$-invariant subgroups.
If $K$ is a $\lambda$-invariant subgroup of $H_{\mu}$, then it must contain at least one of the following elements:
$$a_{1}, \; a_{1}a_{2}, \; a_{1}a_{3}, \; a_{1}a_{4}, \; a_{1}a_{2}a_{3}, \; a_{1}a_{2}a_{4}, \; a_{1}a_{2}a_{6}, \; a_{1}a_{3}a_{4}.$$
If  $a_{1}a_{2}a_{4} \in K$, then $K=K_{\mu}$ and if $a_{1}a_{2}a_{6} \in K$, then $K=K_{\mu}^{*}$. Next, we proceed to check that $K=H_{\mu}$ for the other cases.
(i) If $a_{1} \in K$, then $K=H_{\mu}$.
(ii) If $a_{1}a_{2} \in K$, then $a_{3}a_{4}, a_{6}a_{7}=a_{1}a_{2}a_{3}a_{4}a_{5} \in K$; so $a_{5} \in K$ and $K=H_{\mu}$.
(iii) If $a_{1}a_{3} \in K$, then $a_{2}a_{4}, a_{5}a_{7}=a_{1}a_{3}a_{2}a_{4}a_{6} \in K$; so $a_{6} \in K$ and $K=H_{\mu}$.
(iv) If $a_{1}a_{4} \in K$, then $a_{2}a_{5}, a_{3}a_{6}, a_{4}a_{7}=a_{1}a_{2}a_{5}a_{3}a_{6} \in K$; so $a_{1} \in K$ and $K=H_{\mu}$.
(v) If $a_{1}a_{2}a_{3} \in K$, then $a_{5}a_{6}a_{7}=a_{1}a_{2}a_{3}a_{4} \in K$; so $a_{4} \in K$ and $K=H_{\mu}$.
(vi) If $a_{1}a_{2}a_{5} \in K$, then $a_{3}a_{4}a_{7}=a_{1}a_{2}a_{5}a_{6} \in K$; so $a_{6} \in K$ and $K=H_{\mu}$.
(vii) If $a_{1}a_{3}a_{5} \in K$, then $a_{2}a_{4},a_{6}, a_{3}a_{5}a_{7}=a_{1}a_{2}a_{4}a_{6} \in K$; so $a_{1} \in K$ and $K=H_{\mu}$.
\end{proof}


As the $\lambda$-invariant subgroup 
$$K_{\mu}=\langle a_{1}a_{3}a_{7}, a_{2}a_{3}a_{5},a_{1}a_{2}a_{4}\rangle \cong {\mathbb Z}_{2}^{3}$$
does not contains any of the involutions $a_{j}$, it acts freely on $\widehat{S}_{\mu}$ (the same holds for $K_{\mu}^{*}$).

As a consequence of the Riemann-Hurwitz formula, the quotient $S_{\mu}=\widehat{S}_{\mu}/K_{\mu}$ is a closed Riemann surface of genus seven with $G_{\mu}=H_{\mu}/K_{\mu} \cong {\mathbb Z}_{2}^{3}$ as a group of conformal automorphisms and with $S_{\mu}/G_{\mu}=\widehat{S}_{\mu}/H_{\mu}$ being the orbifold whose underlying Riemann surface is $\widehat{\mathbb C}$ and whose cone points are the points in $B_{\mu}$, each one of order $2$.

Let us denote by $a_{j}^{*}$ the conformal involution of $S_{\mu}$ induced by the involution $a_{j}$, for $j=1,\ldots,7$.

\begin{rema}
Let $\Gamma_{\mu}$ be the Fuchsian group as in remark \ref{observa1}.
As $G_{\mu}$ is abelian group, there is a torsion free finite index subgroup $N_{\mu}$ of $\Gamma_{\mu}$ so that $\Gamma_{\mu}'$ is contained in $N_{\mu}$, $S_{\mu}={\mathbb H}^{2}/N_{\mu}$ and $G_{\mu}=\Gamma_{\mu}/N_{\mu}$. In fact, 
$N_{\mu}=\llangle \Gamma_{\mu}', \alpha_{1}\alpha_{3}\alpha_{7}, \alpha_{2}\alpha_{3}\alpha_{5}, \alpha_{1}\alpha_{2}\alpha_{4}\rrangle_{\Gamma_{\mu}}.$
\end{rema}


\subsection{Algebraic equations for $S_{\mu}$}
Using the equations for $\widehat{S}_{\mu}$ and the explicit group $K_{\mu}$, classical invariant theory permits to obtain a set of equations for $S_{\mu}$. For it, we consider the affine model of $\widehat{S}_{\mu}$, say by taking $x_{7}=1$, which we denote by $\widehat{S}_{\mu}^{0}$. In this affine model the automorphism $a_{7}$ is given by $a_{7}(x_{1},x_{2},x_{3},x_{4},x_{5},x_{6})=(-x_{1},-x_{2},-x_{3},-x_{4},-x_{5},-x_{6})$. A set of generators for the algebra of invariant polynomials in ${\mathbb C}[x_1,x_2,x_3,x_4,x_5,x_6]$ under the natural linear action induced by $K_{\mu}$ is given by
 $$
 \begin{array}{llllllll}
 t_1=x_{1}^{2},& t_2=x_{2}^{2},& t_3=x_{3}^{2},& t_4=x_{4}^{2},& t_5=x_{5}^{2},& t_6=x_{6}^{2},& t_7=x_{1}x_{2}x_{5},& t_8=x_{1}x_{2}x_{3}x_{6},
 \end{array}
 $$
 $$
 \begin{array}{lllll}
 t_{9}=x_{1}x_{4}x_{6},& t_{10}=x_{1}x_{3}x_{4}x_{5},& t_{11}=x_{2}x_{4}x_{5}x_{6},& t_{12}=x_{2}x_{3}x_{4}, &t_{13}=x_{3}x_{5}x_{6}.
 \end{array}
 $$
 
 If we set $$F:\widehat{S}_{\mu}^{0} \to {\mathbb C}^{13}:
 (x_1,x_2,x_3,x_4,x_5,x_6) \mapsto (t_{1}, t_{2}, t_{3}, t_{4}, t_{5}, t_{6}, t_{7}, t_{8},t_{9},t_{10},t_{11},t_{12},t_{13}),$$
 then $F(\widehat{S}_{\mu}^{0})=S_{\mu}^{0}$ will provide the following affine model for $S_{\mu}$:
 
 \begin{equation}\label{eq1}
S_{\mu}^{0}:=  \left\{ \begin{array}{llll}
 t_1+t_2+t_3=0,& \mu_{4}\; t_1+t_2+t_4=0, & \mu_{5}\; t_1+t_2+t_5=0,& \mu_{6}\; t_1+t_2+t_6=0,\\
\mu_{7}\; t_1+t_2+1=0,& t_{6}t_{10} = t_{9}t_{13}, & t_{6}t_{7}t_{12} = t_{8}t_{11},& t_{5}t_{9}t_{12} = t_{10}t_{11},\\
t_{5}t_{8} = t_{7}t_{13}, & t_{5}t_{6}t_{12} = t_{11}t_{13}, & t_{4}t_{8} = t_{9}t_{12}, & t_{4}t_{7}t_{13} = t_{10}t_{11},\\
t_{4}t_{6}t_{7} = t_{9}t_{11}, & t_{3}t_{11} = t_{12}t_{13}, & t_{3}t_{6}t_{7} = t_{8}t_{13}, & t_{3}t_{5}t_{9} = t_{10}t_{13},\\
t_{3}t_{5}t_{6} = t_{13}^{2}, & t_{3}t_{4}t_{7} = t_{10}t_{12}, & t_{2}t_{10} = t_{7}t_{12}, & t_{2}t_{9}t_{13} = t_{8}t_{11}, \\
t_{2}t_{5}t_{9} = t_{7}t_{11}, & t_{2}t_{4}t_{13} = t_{11}t_{12}, & t_{2}t_{4}t_{5}t_{6} = t_{11}^{2}, & t_{2}t_{3}t_{9} = t_{8}t_{12},\\
t_{2}t_{3}t_{4} = t_{12}^{2}, & t_{1}t_{12}t_{13} = t_{8}t_{10}, & t_{1}t_{11} = t_{7}t_{9}, & t_{1}t_{6}t_{12} = t_{8}t_{9},\\
t_{1}t_{5}t_{12} = t_{7}t_{10}, & t_{1}t_{4}t_{13} = t_{9}t_{10}, & t_{1}t_{4}t_{6} = t_{9}^{2}, & t_{1}t_{3}t_{4}t_{5} = t_{10}^{2},\\
t_{1}t_{2}t_{13} = t_{7}t_{8}, & t_{1}t_{2}t_{5} = t_{7}^{2}, & t_{1}t_{2}t_{3}t_{6} = t_{8}^{2}.
 \end{array}
 \right\}
 \subset {\mathbb C}^{13}
 \end{equation}

Of course, one may see that the variables $t_{2}$, $t_{3}$, $t_{4}$, $t_{5}$ and $t_{6}$ are uniquely determined by the variable $t_{1}$. Other variables can also be determined in order to get a lower dimensional model. The group $G_{\mu}$ is, in this model, generated by the following three involutions:
$$(t_{1},\ldots,t_{13}) \mapsto (t_{1},t_{2},t_{3},t_{4},t_{5},t_{6},-t_{7},-t_{8},-t_{9},-t_{10},t_{11},t_{12},t_{13}),$$
$$(t_{1},\ldots,t_{13}) \mapsto (t_{1},t_{2},t_{3},t_{4},t_{5},t_{6},-t_{7},-t_{8},t_{9},t_{10},-t_{11},-t_{12},t_{13}),$$
$$(t_{1},\ldots,t_{13}) \mapsto (t_{1},t_{2},t_{3},t_{4},t_{5},t_{6},t_{7},-t_{8},t_{9},-t_{10},t_{11},-t_{12},-t_{13}).$$

\subsection{Some elliptic curves associated to $S_{\mu}$}
Let us observe that 
$$G_{\mu}=\langle a_{1}^{*},a_{2}^{*},a_{3}^{*}\rangle$$ 
and that 
$$a_{4}^{*}=a_{1}^{*}a_{2}^{*}, \; a_{5}^{*}=a_{2}^{*}a_{3}^{*}, \; a_{6}^{*}=a_{1}^{*}a_{2}^{*}a_{3}^{*}, \; a_{7}^{*}=a_{1}^{*}a_{3}^{*}.$$

In this way, each of the seven involutions $a_{r}^{*} \in G_{\mu}$ is induced by exactly one of the involutions of $H_{\mu}$ with fixed points and, in particular, it acts with exactly $4$ fixed points on $S_{\mu}$. By the Riemann-Hurwitz formula, the quotient orbifold $S_{\mu}/\langle a_{r}^{*} \rangle$ has genus three and exactly four cone points, each of order $2$ (over that orbifold there is the conformal action of the group $G_{\mu}/\langle a_{r}^{*}\rangle \cong {\mathbb Z}_{2}^{2}$ for which the cone points form one orbit).

If $a_{i}^{*}, a_{j}^{*} \in G_{\mu}$ are two different involutions and $L_{ij}=\langle a_{i}^{*}, a_{j}^{*}  \rangle \cong {\mathbb Z}_{2}^{2}$, then (by the Riemann-Hurwitz formula) the quotient orbifold $S_{\mu}/L_{ij}$ is a closed Riemann surface $T_{ij}$ of genus one with exactly six cone points, all of them of order $2$. Further, $L^{*}_{ij}=G_{\mu}/L_{ij} \cong {\mathbb Z}_{2}$ acts as a group of automorphisms of $T_{ij}$ permuting these six cone points (and fixing none of them).
These six cone points are projected onto three of the cone points of $S_{\mu}/G_{\mu}=T_{ij}/L^{*}_{ij}$; they being $\mu_{i}$, $\mu_{j}$ and $\mu_{r}$, where $a_{r}^{*}=a_{i}^{*} a_{j}^{*}$ and $\mu_{1}=\infty$, $\mu_{2}=0$ and $\mu_{3}=1$. In this way, 
 $$T_{ij}: \; y^{2}=\prod_{k \notin \{i,j,r\}} (x-\mu_{k})$$
 where, in the case $k=1$ the factor $(x-\mu_{1})$ is deleted from the previous expression.
 
 Let us observe that 
there are exactly seven different subgroups $L_{ij}$ of $G_{\mu}$, these being given by
$$L_{12},\;  L_{13},\;  L_{15},\; L_{23},\; L_{1,7},\; L_{34},\; L_{47},$$
so there are only seven elliptic curves as above; 

$$\begin{array}{lll}
T_{23}: & y_{1}^{2}=(x-\mu_{4})(x-\mu_{6})(x-\mu_{7}),\\
T_{12}: & y_{2}^{2}=(x-1)(x-\mu_{5})(x-\mu_{6})(x-\mu_{7}),\\
T_{13}: & y_{3}^{2}=x(x-\mu_{4})(x-\mu_{5})(x-\mu_{6}),\\
T_{15}: &  y_{4}^{2}=x(x-1)(x-\mu_{4})(x-\mu_{7}),\\
T_{26}: & y_{5}^{2}=(x-1)(x-\mu_{4})(x-\mu_{5}), \\
T_{34}: & y_{6}^{2}=x(x-\mu_{5})(x-\mu_{7}), \\
T_{47}: & y_{7}^{2}=x(x-1)(x-\mu_{6}).
\end{array}
$$

It can be seen that another equation for $S_{\mu}$ is provided by the fiber product of the three elliptic curves $T_{23}, T_{12}$ and $T_{13}$ as follows

\begin{equation}\label{eq3}
\left\{\begin{array}{lcl}
y_{1}^{2}=(x-\mu_{4})(x-\mu_{6})(x-\mu_{7})\\
\\
y_{2}^{2}=(x-1)(x-\mu_{5})(x-\mu_{6})(x-\mu_{7})\\
\\
y_{3}^{2}=x(x-\mu_{4})(x-\mu_{5})(x-\mu_{6})
\end{array}
\right\} \subset {\mathbb C}^{4}.
\end{equation}

In this model, the group $G_{\mu}$ is generated by the involutions 
$$(x,y_{1},y_{2},y_{3}) \mapsto (x,-y_{1},y_{2},y_{3}),$$
$$(x,y_{1},y_{2},y_{3}) \mapsto (x,y_{1},-y_{2},y_{3}),$$
$$(x,y_{1},y_{2},y_{3}) \mapsto (x,y_{1},y_{2},-y_{3}).$$

\subsection{An elliptic decomposition of the  jacobian variety $JS_{\mu}$}
If $\mu=(\mu_{4},\mu_{5},\mu_{6},\mu_{7}) \in \Omega$, then 
in \cite{CHQ} we obtained that the jacobian variety $J\widehat{S}_{\mu}$ of $\widehat{S}_{\mu}$ is isogenous to a product
$$E_{1,\mu} \times \cdots \times E_{35,\mu} \times JX_{1,\mu} \cdots \times JX_{7,\mu},$$
where each of the $E_{j,\mu}$ is an elliptic curve (corresponding to the choice of $4$ points in $B_{\mu}$) and each of the $X_{j,\mu}$ is a genus two Riemann surface (these are provided by the choice of $6$ points in $B_{\mu}$). On the other hand, $J\widehat{S}_{\mu}$ is also isogenous to a product $JS_{\mu} \times {\mathcal P}_{\mu}$, where ${\mathcal P}_{\mu}$ is a polarized abelian variety of dimension $42$. In particular, one obtains that 
 $$E_{1,\mu} \times \cdots \times E_{35,\mu} \times JX_{1,\mu} \cdots \times JX_{7,\mu} \sim JS_{\mu} \times {\mathcal P}_{\mu}.$$ 

Unfortunately, the above isogeny does not provide more expliciy information on $JS_{\mu}$. In order to get a decomposition of it we proceed as follows. First, observe that, for the seven subgroups $L_{12},  L_{13},  L_{15}, L_{23}, L_{1,7}, L_{34}, L_{47}$, it holds the following:
(i) $L_{ij}L_{kr}=L_{kr}L_{ij}$, (ii) each quotient $S_{\mu}/L_{ij}$ has genus one and (iii) $S_{\mu}/L_{ij}L_{kr}=S_{\mu}/G_{\mu}$ has genus zero. So, from Kani-Rosen decomposition theorem \cite{K-R}, one obtains the following.

\begin{theo}\label{isogeno}
For each $\mu \in \Omega$, it holds that
$$JS_{\mu} \sim T_{12} \times T_{13} \times T_{15} \times T_{23} \times T_{26} \times T_{34} \times T_{47}.$$
\end{theo}


The above asserts that  the jacobian variety $JS_{\mu}$ of $S_{\mu}$ is completely decomposable (see also \cite{CHQ}).
Observe that the seven elliptic factors are some of the $35$ elliptic factors of $J\widehat{S}_{\mu}$.

\subsection{A remark about $K^{*}_{\mu}$}
As $K^{*}_{\mu}$ also acts freely on $\widehat{S}_{\mu}$, we may consider the Riemann surface $S_{\mu}^{*}=\widehat{S}_{\mu}/K^{*}_{\mu}$, which admits the group ${\mathbb Z}_{2}^{3} \cong G_{\mu}^{*}=H_{\mu}/K^{*}_{\mu}=\langle b_{1},b_{2},b_{3}\rangle$, where $b_{j}$ is induced by $a_{j}$. In this case, $$b_{4}=b_{1}b_{3}, \; b_{5}=b_{1}b_{2}b_{3}, \; b_{6}=b_{1}b_{2}, \; b_{7}=b_{2}b_{3}.$$

Working in a similar fashion as for $K_{\mu}$, one can obtain the following set of equations for $S_{\mu}^{*}$
\begin{equation}\label{eq3}
\left\{\begin{array}{lcl}
w_{1}^{2}=(z-\mu_{4})(z-\mu_{5})(z-\mu_{6})\\
\\
w_{2}^{2}=(z-\mu_{4})(z-\mu_{5})(z-\mu_{6})\\
\\
w_{3}^{2}=z(z-1)(z-\mu_{4})(z-\mu_{6})
\end{array}
\right\} \subset {\mathbb C}^{4}.
\end{equation}

In this model, the group $G_{\mu}^{*}$ is again generated by the involutions 
$$(z,w_{1},w_{2},w_{3}) \mapsto (z,-w_{1},w_{2},w_{3}),$$
$$(z,w_{1},w_{2},w_{3}) \mapsto (z,w_{1},-w_{2},w_{3}),$$
$$(z,w_{1},w_{2},w_{3}) \mapsto (z,w_{1},w_{2},-w_{3}).$$

Direct computations permits to see that $$S_{\mu} \cong S_{\widetilde{\mu}}^{*}$$
when $$\widetilde{\mu}=\left( \frac{\mu_{6}}{\mu_{5}}, \frac{\mu_{7}}{\mu_{5}}, \frac{\mu_{4}}{\mu_{5}},\frac{1}{\mu_{5}} \right).$$

In fact, just use the change of variable induced by the transformation $T(x)=x/\mu_{5}$ to $S_{\mu}$ to obtain $S^{*}_{\widetilde{\mu}}$.

\section{The Fricke-Macbeath curve}\label{Sec:FM}
In this section we observe that the Fricke-Macbeath curve is one of the members of the previous $4$-dimensional family $S_{\mu}$ of genus seven Riemann surfaces.

First, as already observed in the introduction, the Fricke-Macbeath curve $S$ admits a group $K = G \rtimes \langle T \rangle \cong {\mathbb Z}_{2}^{3} \rtimes {\mathbb Z}_{7}$ of conformal automorphisms. Let us consider the orbifold $S/G$, which has some number of cone points, each one of order $2$. Moreover, the automorphism $T$ induces an automorphism $A_{T}$ of $S/G$, of order seven and keeping invariant these cone points.

\begin{prop}
The quotient orbifold $S/G$ is the Riemann sphere with exactly seven cone points of order $2$.
\end{prop}
\begin{proof}
The automorphism $T$ must permute the seven elements of order two in $G$. There are only two possibilities: (i) these seven elements form one orbit or (ii) each element of order two commutes with $T$. The quotient orbit $S/G$, by the Riemann-Hurwitz formula, is either a torus with exactly three cone points of order $2$ or the Riemann sphere $\widehat{\mathbb C}$ with exactly $7$ cone points of order $2$. If $S/G$ has genus one, as there are only three cone points and such a set is invariant under $A_{T}$, it follows that $A_{T}$ must have fixed points; a contradiction as there is no genus one Riemann surface admitting a conformal automorphism of order seven with fixed points. It follows that $S/G$ is the Riemann sphere (with exactly seven cone points of order $2$). In this last situation, $A_{T}$ only has two fixed points, so $T$ only satisfies condition (i) above.
\end{proof}


Let us consider a regular branched cover $Q:S \to \widehat{\mathbb C}$ whose deck group is $G \cong {\mathbb Z}_{2}^{3}$. Up to postcomposing $Q$ by a suitable M\"obius transformation, we may assume that $A_{T}(z)=\rho z$, where $\rho=e^{2 \pi i/7}$. As the seven cone points of order $2$ are cyclically permuted by $A_{T}$, we may also assume that they are given by the points
$$\lambda_{1}=1,\; \lambda_{2}=\rho,\; \lambda_{3}=\rho^{2},\; \lambda_{4}=\rho^{3},\; \lambda_{5}=\rho^{4},\; \lambda_{6}=\rho^{5},\; \lambda_{7}=\rho^{6}.$$

Let us observe that 
$${\mu}^{0}=(A(\rho^{3}),A(\rho^{4}),A(\rho^{5}),A(\rho^{6})) \in \Omega.$$

As $S/G={\widehat S}_{{\mu}^{0}}/H_{{\mu}^{0}}$ and $G$ is abelian (recall that $\widehat{S}_{{\mu}^{0}}$ is the highest abelian branched cover),  
by classical covering theory, there should be a subgroup $K<H_{\mu^{0}}$, $K\cong {\mathbb Z}_{2}^{3}$, acting freely on $\widehat{S}_{\mu^{0}}$ so that there is an isomorphism
$\phi:S \to \widehat{S}_{\mu^{0}}/K$ with $\phi G \phi^{-1}=H_{\mu^{0}}/K$.

Let us also observe that the rotation $A_{T}(z)=\rho z$ lifts under $P_{\mu^{0}}$ to an automorphism
$\widehat{T}$ of $\widehat{S}_{\mu^{0}}$ of order $7$ of the form \cite{GHL} 
$$\widehat{T}([x_{1}:\cdots:x_{7}])=[x_{7}:c_{2}x_{1}:x_{2}:c_{4}x_{3}:c_{5}x_{4}:c_{6}x_{5}:c_{7}x_{6}],$$
where
$$c_{2}=\sqrt{A(\rho^{6})},\; c_{4}=i\sqrt{A(\rho^{3})},\; c_{5}=i\sqrt{A(\rho^{4})},\; c_{6}=i\sqrt{A(\rho^{5})},\;c_{7}=i\sqrt{A(\rho^{6})}.$$

One has that $\widehat{T} a_{j} \widehat{T}^{-1}=a_{j+1}$, for $j=1,...,6$ and $\widehat{T} a_{7} \widehat{T}^{-1}=a_{1}$. It follows that, for this explicit value of $\mu=\mu^{0}$, one has that $\lambda=\Lambda_{\widehat{T}}$.

The subgroup $K$ above must satisfy that $\widehat{T} K \widehat{T}^{-1}=K$ (since $R$ also lifts to the Fricke-Macbeath curve $S$), that is, $K$ is a $\lambda$-invariant subgroup of $H_{\mu^{0}}$; it follows (by Lemma \ref{rhoinv}) that 
$K \in \{K_{\mu^{0}},K^{*}_{\mu^{0}}\}$. But, for this value of $\mu=\mu^{0}$, the automorphism $\zeta$ is induced by an automorphism of $\widehat{S}_{\mu^{0}}$ given by
$$U([x_{1}:\cdots:x_{7}])=[x_{2}:d_{2}x_{1}:x_{7}:d_{4}x_{6}:d_{5}x_{5}:d_{6}x_{4}:d_{7}x_{3}],$$
where
$$d_{2}=d_{7}=\sqrt{A(\rho^{6})}, \; d_{4}=\sqrt{A(\rho^{3})},\; d_{5}=\sqrt{A(\rho^{4})}, \; d_{6}=\sqrt{A(\rho^{5})}.$$

In particular, both genus seven Riemann surfaces unifomized by $K_{\mu^{0}}$ and $K_{\mu^{0}}^{*}$ are isomorphic. In this way, we may assume that 
$$K=K_{\mu^{0}}=\langle a_{1}a_{3}a_{7}, a_{2}a_{3}a_{5},a_{1}a_{2}a_{4}\rangle.$$

\begin{rema}
Another way to see that we may assume $K=K_{\mu^{0}}$ is as follows.
The subgroup $$K_{\mu^{0}}=\langle a_{1}a_{3}a_{7}, a_{2}a_{3}a_{5},a_{1}a_{2}a_{4}\rangle \cong {\mathbb Z}_{2}^{3}$$ acts freely on $\widehat{S}_{\mu^{0}}$ and it is normalized by  $\widehat{T}$. In particular, 
$S^{*}=\widehat{S}_{\mu^{0}}/K_{\mu_{0}}$ is a closed Riemann surface of genus $7$ admitting the group 
$L=H_{\mu^{0}}/K_{\mu^{0}} \cong {\mathbb Z}_{2}^{3}$ as a group of conformal automorphisms. It also admits an automorphism $T^{*}$ of order $7$ (induced by $\widehat{T}$) permuting cyclically the involutions $a_{j}^{*}$. As $S^{*}/\langle L, T^{*}\rangle=\widehat{S}_{\mu^{0}}/\langle H_{\mu^{0}},\widehat{T}\rangle$ has signature $(0;2,7,7)$, and the Fricke-Macbeath curve is the only closed Riemann surface of genus seven, up to isomorphisms, admitting a group of automorphisms isomorphic to ${\mathbb Z}_{2}^{3} \rtimes {\mathbb Z}_{7}$ with corresponding quotient orbifold of the previous signature, we obtain the desired result. 
\end{rema}


Summarizing all the above is the following.

\begin{theo}
The Fricke-Macbeath curve is isomorphic to 
the Riemann surface $S_{\mu^{0}}$, where
$$\mu^{0}=\left(\frac{(1+\rho)(\rho^{2}-1)}{\rho^{3}-1},\frac{(1+\rho)(\rho^{3}-1)}{\rho^{4}-1},\frac{(1+\rho)(\rho^{4}-1)}{\rho^{5}-1},\frac{(1+\rho)(\rho^{5}-1)}{\rho^{6}-1}\right).$$

In particular, the Fricke-Macbeath curve is given by the fiber product equation 
\begin{equation}\label{eq00}
\left\{\begin{array}{lcl}
z_{1}^{2}=(u-A(\rho^{3}))(u-A(\rho^{5}))(u-A(\rho^{6}))\\
\\
z_{2}^{2}=(u-1)(u-A(\rho^{4}))(u-A(\rho^{5}))(u-A(\rho^{6}))\\
\\
z_{3}^{2}=u(u-A(\rho^{3}))(u-A(\rho^{4}))(u-A(\rho^{5}))
\end{array}
\right\} \subset {\mathbb C}^{4}.
\end{equation}
\end{theo}

\begin{rema}
It can be seen that the fiber product equation \eqref{eq00} is isomorphic to Wiman's fiber product \eqref{eq0} using the change of variable $u=A(x)$.
\end{rema}


In this case, i.e., for $\mu=\mu^{0}$, all the seven elliptic curves $T_{ij}$ are isomorphic to 
$$E:= y^{2}=x(x-1)\left(x-\frac{(1+\rho)(\rho^{5}-1)}{\rho^{6}-1}\right),$$
in particular, as a consequence of Theorem \ref{isogeno}, we obtain the following well known fact.

\begin{coro}
If $S$ is the Fricke-Macbeath curve, then
$$JS=JS_{\mu^{0}} \sim E^{7}.$$
\end{coro}

\begin{rema}
We also note that 
$$J\widehat{S}_{\mu^{0}} \sim E_{1}^{7} \times E_{2}^{7} \times E_{3}^{7} \times E_{4}^{7} \times E_{5}^{7} \times E_{6}^{14}$$
where 
$$E_{1}: \; y^{2}=(x-1)(x-\rho)\left(x-\rho^{2}\right)\left(x-\rho^{3}\right),$$
$$E_{2}: \; y^{2}=(x-1)(x-\rho)\left(x-\rho^{2}\right)\left(x-\rho^{4}\right),$$
$$E_{3}: \; y^{2}=(x-1)(x-\rho)\left(x-\rho^{2}\right)\left(x-\rho^{5}\right),$$
$$E_{4}: \; y^{2}=(x-1)(x-\rho)\left(x-\rho^{3}\right)\left(x-\rho^{4}\right),$$
$$E_{5}: \; y^{2}=(x-1)(x-\rho)\left(x-\rho^{4}\right)\left(x-\rho^{5}\right),$$
$$E_{6}: \; y^{2}=\left(x-1\right)\left(x-\cos(2\pi/7)\right)\left(x-\cos(4\pi/7)\right)\left(x-\cos(6\pi/7)\right).$$

Also, observe that $E_{3} \cong E$.
\end{rema}

\subsection*{Acknowledgments}
The author is grateful to the referee whose suggestions, comments and corrections done to the preliminary versions helped to improve the presentation of the paper. I also want to thanks I. Cheltsov for the reference to the work of Y. Prokhorov.


\end{document}